\documentclass[11pt]{article}
\usepackage{amsmath,amsfonts,amsthm,amssymb, mathtools}
\usepackage{mathrsfs, graphicx,color,latexsym, tikz, calc,
}
\usepackage{tikz-cd}
\usepackage{graphicx}
\usepackage{epstopdf}
\usepackage{enumerate}
\usepackage{caption}
\usepackage{stmaryrd}
\usepackage{hyperref}

\usetikzlibrary{shadows}
\usetikzlibrary{patterns,arrows,decorations.pathreplacing}
\usepackage{ulem}
\newtheorem{theorem}{Theorem}[section]
\newtheorem{lemma}[theorem]{Lemma}

\newtheorem{corollary}[theorem]{Corollary}

\allowdisplaybreaks[4]

\title{\bf \Large }
\date{\today} 
\title{\bf \Large  A result for hemi-bundled  cross-intersecting families\footnote{This paper was published on 
Advances in Applied Mathematics 169 (2025) 102912. 
E-mail addresses: \url{wuyjmath@163.com} (Y. Wu), \url{fenglh@163.com} (L. Feng), 
\url{ytli0921@hnu.edu.cn} (Y. Li)}}
\author{
{\small  Yongjiang Wu, \ \ Lihua Feng, \ \ Yongtao Li\footnote{Corresponding author} }\\[2mm]
\small School of Mathematics and Statistics, HNP-LAMA, Central South University\\
 \small Changsha, Hunan, 410083, China}

\begin{document}
\maketitle
\begin{abstract}
Two families $\mathcal{F}$ and $\mathcal{G}$ are called 
cross-intersecting if for every $F\in \mathcal{F}$ and $G\in \mathcal{G}$, the intersection $F\cap G$ is non-empty. 
It is significant to determine the maximum sum of sizes of cross-intersecting families 
under the additional assumption that one of the two families is intersecting. Such a pair of families is said to be hemi-bundled. 
In particular, Frankl (2016) proved that for $k \geq 1, t\ge 0$ and $n \geq 2 k+t$, 
 if $\mathcal{F} \subseteq\binom{[n]}{k+t}$ and $\mathcal{G} \subseteq\binom{[n]}{k}$ are cross-intersecting families, in which  $\mathcal{F}$ is non-empty and $(t+1)$-intersecting, then
$|\mathcal{F}|+|\mathcal{G}| \leq\binom{n}{k}-\binom{n-k-t}{k}+1$. 
This bound is achieved when $\mathcal{F}$ consists of a single set. 
In this paper, we generalize this result under the constraint  $|\mathcal{F}| \geq r$ for every $r\leq n-k-t+1$.  
Moreover, we investigate the stability results of Katona's theorem for non-uniform families with the $s$-union property. Our result extends the stabilities established by Frankl (2017) and Li and Wu (2024). 
As applications, we revisit a recent result of Frankl and Wang (2024) as well as a result of Kupavskii (2018). Furthermore, we determine the extremal families in these two results. 
\end{abstract} 



\section{Introduction}

For integers $a\leq b$, let  $[a, b]=\{a, a+1, \ldots, b\}$ and let $[n]=[1,n]$ for short. Let $2^{[n]}$ denote the power set of $[n]$. For $0 \leq k \leq n$, let $\binom{[n]}{k}$ denote the collection of all $k$-element subsets of $[n]$. A family $\mathcal{F} \subseteq 2^{[n]}$ is called $k$-\textit{uniform} if $|F| = k$ 
for every $F\in \mathcal{F}$.  Two families  $\mathcal{F}, \mathcal{G}\subseteq 2^{[n]}$ are  said to be \textit{isomorphic} if there exists a permutation $\sigma$ on $[n]$ such that $\mathcal{G}=\left\{\{\sigma(x):x\in F\}: F\in \mathcal{F}\right\}$, and we write $\mathcal{F}\cong\mathcal{G}$.  
A family $\mathcal{F} \subseteq 2^{[n]}$ is called $t$-\textit{intersecting} if $|F\cap F^{\prime}|\geq t$
 for all $F, F^{\prime}\in \mathcal{F}$. If $t=1$, then $\mathcal{F}$ is simply called \textit{intersecting}.

 Erd\H{o}s, Ko and Rado \cite{E61} determined the maximum $k$-uniform
intersecting family by proving that if $k\ge 2$, $n\ge 2k$ and 
 $\mathcal{F} \subseteq\binom{[n]}{k}$ is an intersecting family, then 
\begin{equation}
    \label{eq-EKR}
    |\mathcal{F}|\leq\binom{n-1}{k-1}.  
\end{equation}
For $n > 2 k$, the equality holds if and only if $\mathcal{F}=\big\{F\in \binom{[n]}{k}: x\in F\big\}$ for some $x \in [n]$. Such a family is
called a \textit{full star}.
There are many different proofs and methods for the Erd\H{o}s--Ko--Rado theorem; see, e.g., the Katona cycle method \cite{Kat1972},  the probabilistic method \cite{AS2016}, the algebraic methods \cite{Fur2006,HZ2017,Lov1979} and other combinatorial methods \cite{Day1974b,FF2012,H18,K64,KZ2018}.

\subsection{Hemi-bundled cross-intersecting families} 

Two families  $\mathcal{F},\mathcal{G}\subseteq 2^{[n]}$ are  called \textit{cross-intersecting} if $|F\cap G|\geq 1$
 for any $F\in \mathcal{F}$ and $G\in \mathcal{G}$. In 1967, Hilton and Milner \cite{H67} proved that if $n \geq 2 k$ and  $\mathcal{F},\mathcal{G} \subseteq\binom{[n]}{k}$ are non-empty cross-intersecting families. Then
\begin{equation}
    \label{eq-HM}
    |\mathcal{F}|+|\mathcal{G}| \leq\binom{n}{k}-\binom{n-k}{k}+1. 
\end{equation}
This result initiated the study to find the maximum of the sum of sizes of cross-intersecting families.  
In 1992, Frankl and Tokushige \cite{FT92} established the following extension: Let $ 2\le \ell \le k $ and $n\geq k+\ell$. If $\mathcal{F} \subseteq\binom{[n]}{k}$ and $\mathcal{G} \subseteq\binom{[n]}{\ell}$ are non-empty cross-intersecting families, then 
\begin{equation}
    \label{eq-FT}
     |\mathcal{F}|+|\mathcal{G}| \leq\binom{n}{k}-\binom{n-\ell}{k}+1. 
\end{equation} 
This bound can be attained {  when $|\mathcal{G}|=1$}. 
{  A family $\mathcal{F}$ is called  \textit{trivial} if
it is contained in a full star, i.e., all sets of $\mathcal{F}$ contain a common element; otherwise it is called \textit{non-trivial}}.
Frankl and Tokushige \cite{FT1998} (see  Frankl \cite{Fra2024}) proved that 
for $2\le \ell \le k$ and $n\ge k+\ell $,  
suppose that $\mathcal{F}\subseteq {[n] \choose k}$ 
and $\mathcal{G} \subseteq {[n] \choose \ell }$ are non-empty cross-intersecting. If $\mathcal{F}$ and $\mathcal{G}$ 
are non-trivial, then 
\[  |\mathcal{F}|+|\mathcal{G}| 
\le  {n \choose k} - 2{n-\ell \choose k} + {n-2\ell \choose k} +2.  \] 
This bound can be achieved when $\mathcal{G}$ consists of two disjoint sets. In the case $\ell < k$, 
this bound is also valid under a weaker condition, that is, $\mathcal{F}$ is non-empty and $\mathcal{G}$ is non-trivial; see Frankl and Wang \cite{FW2024-cover3}. 
Apart from the well-known results above, 
there are a large number of generalizations and extensions of Hilton--Milner's result (\ref{eq-HM}) in the literature; see, e.g.,  \cite{BF2022,CLLW2022,F16,FK2017-CPC,FW2024,HP2025-jcta,SFQ2022,WZ2013}.

In this paper, we study the problem of finding the maximum sum of sizes of cross-intersecting families {\it under a certain condition that one of them is intersecting}. 
Such a pair of cross-intersecting families is usually called {\it hemi-bundled}.  
It is worth mentioning that some extremal problems on hemi-bundled families are  investigated in the literature; see 
Scott and Wilmer \cite{SW2021} and Yu et al. \cite{YKXZG2022} for variants of the Bollob\'{a}s two families theorem. 
Particularly, the problem on hemi-bundled cross-intersecting families was initially studied by Frankl \cite{F16}.

\begin{theorem}[Frankl \cite{F16}] 
\label{F16}
 Let $k \geq 2,t\ge 0$ and $n \geq 2 k+t$ be integers.  Let $\mathcal{F} \subseteq\binom{[n]}{k+t}$ and $\mathcal{G} \subseteq\binom{[n]}{k}$ be  cross-intersecting families. If $\mathcal{F}$ is non-empty and 
 $(t+1)$-intersecting, then
$$
|\mathcal{F}|+|\mathcal{G}| \leq\binom{n}{k}-\binom{n-k-t}{k}+1.
$$
For $n > 2 k+t$, the equality holds if and only if $\mathcal{F}=\left\{F_1\right\}$ for some $F_1\in\binom{[n]}{k+t}$ and $\mathcal{G}=\left\{G \in\binom{[n]}{k}: G \cap F_1 \neq \emptyset \right\}$, or $k=2$, there is one more possibility, namely, 
$\mathcal{F}=\{F\in {[n] \choose t+2}: [t+1] \subseteq F\}$ and $\mathcal{G}=\left\{G \in\binom{[n]}{k}: G \cap[t+1] \neq \emptyset\right\}$ under isomorphism.
\end{theorem}

Recently, the third author and Wu \cite[Theorem 15]{L24} sharpened Theorem \ref{F16} in the case $t=1$ during the study on {  the stabilities of Katona's theorem (see Theorem \ref{thm-K64})}.
Subsequently, Wu \cite{W23} showed the following extension under the constraint $|\mathcal{F}|\ge 2$.

\begin{theorem}[Wu \cite{W23}] 
\label{W23}
 Let $k \geq 3,t\ge 0$ and $n \geq 2 k+t$ be positive integers. Let $\mathcal{F} \subseteq\binom{[n]}{k+t}$ and $\mathcal{G} \subseteq\binom{[n]}{k}$ be cross-intersecting families. If $\mathcal{F}$ is $(t+1)$-intersecting and $|\mathcal{F}| \geq 2$, then
$$
|\mathcal{F}|+|\mathcal{G}| \leq\binom{n}{k}-\binom{n-k-t}{k}-\binom{n-k-t-1}{k-1}+2.
$$
For $n > 2 k+t$, the equality holds if and only if $\mathcal{F}=\left\{F_1, F_2\right\}$ for some $F_1, F_2 \in\binom{[n]}{k+t}$ with $\left|F_1 \cap F_2\right|=k+t-1$ and $\mathcal{G}=\left\{G \in\binom{[n]}{k}: G \cap F_1 \neq \emptyset \text{ and } G \cap F_2 \neq \emptyset\right\}$, or two more possibilities when $k=3$, namely, $\mathcal{F}=\{F\in {[n] \choose t+3}: [t+2]\subseteq F\}$ and $ \mathcal{G}=\left\{G \in\binom{[n]}{3}: G \cap[t+2] \neq \emptyset\right\}$, or $\mathcal{F}=\left\{F\in\binom{[n]}{t+3}: [t+1] \subseteq F\right\}$ and $ \mathcal{G}=\left\{G \in\binom{[n]}{3}: G \cap[t+1] \neq \emptyset\right\}$ under isomorphism.
\end{theorem}

\iffalse 
We point out that in the case $k=3$, 
there missed an extremal family in \cite{W23}, namely, $\mathcal{F}=\big\{F\in\binom{[n]}{t+3}: [t+1] \subseteq F\big\} $ and $\mathcal{G}=\big\{G \in\binom{[n]}{3}: G \cap[t+1] \neq \emptyset\big\}$. 
\fi 

Motivated by Theorems \ref{F16} and \ref{W23}, we investigate the general case for a family $\mathcal{F}$ with  $|\mathcal{F}| \geq r$. 
The first main result of this paper extends these results as follows. 

\begin{theorem}[Main result] \label{main1}
 Let $k \geq 2, t\ge 0$, $r\geq 1$ and $n \geq 2 k+t$ be integers.  Let $\mathcal{F} \subseteq\binom{[n]}{k+t}$ and $\mathcal{G} \subseteq\binom{[n]}{k}$ be  cross-intersecting families, where  $\mathcal{F}$ is $(t+1)$-intersecting and $|\mathcal{F}| \ge r$.  \\ 
{\rm (i)}
If $ r \leq k-1$, then
$$
|\mathcal{F}|+|\mathcal{G}| \leq\binom{n}{k}-\binom{n-k-t+1}{k}+\binom{n-k-t-r+1}{k-r}+r.
$$
For $n > 2 k+t$, the equality holds if and only if $\mathcal{F}=\left\{F_1,\ldots,F_r \right\}$ with 
$F_i:=[k+t-1]\cup \{k+t-1+i\}$ for each $i\in [r]$  
and $\mathcal{G}=\left\{G \in\binom{[n]}{k}: G \cap F_i \neq\emptyset\text{~for each~} i\in [r]\right\}$, or when $k=r+1$, $\mathcal{F}=\Big\{[t+r] \cup\{i\}: i \in\{t+r+1, \ldots, n\} \Big\}$ and $\mathcal{G}=\left\{G \in\binom{[n]}{r+1}: G \cap[t+r] \neq \emptyset\right\}$, or when $k=3$, $\mathcal{F}=\left\{F\in\binom{[n]}{t+3}: [t+1] \subseteq F\right\}$ and $\mathcal{G}=\left\{G \in\binom{[n]}{3}: G \cap[t+1] \neq \emptyset\right\}$ under isomorphism. \\ 
{\rm (ii)}
If $k\leq r \leq n-k-t+1$, then
$$
|\mathcal{F}|+|\mathcal{G}| \leq\binom{n}{k}-\binom{n-k-t+1}{k}+n-k-t+1.
$$
For $n > 2 k+t$, the equality holds if and only if $\mathcal{F}=\Big\{[k+t-1] \cup\{i\}: i \in\{k+t, \ldots, n\} \Big\}$ and $\mathcal{G}=\left\{G \in\binom{[n]}{k}: G \cap[k+t-1] \neq \emptyset\right\}$, or when $k=3$, $\mathcal{F}=\left\{F\in\binom{[n]}{t+3}: [t+1] \subseteq F\right\}$ and $\mathcal{G}=\left\{G \in\binom{[n]}{3}: G \cap[t+1] \neq \emptyset\right\}$ under isomorphism. 
\end{theorem}

We shall show that our result 
on hemi-bundled cross-intersecting families 
is extremely useful for the study of extremal set problems on cross-intersecting families and  
 intersecting families with small diversity. 
 We defer the detailed discussions to Subsection \ref{sec-1.3}.  
In next section, with the help of Theorem \ref{main1}, 
we investigate the families with some certain union properties.

\subsection{More stability results on  Katona's theorem}  

\label{sec-Stability-Katona}

A family $\mathcal{F} \subseteq 2^{[n]}$ is called $s$-\textit{union} if $|F\cup F^{\prime}|\leq s$ for all $F, F^{\prime}\in \mathcal{F}$.
Obviously, $\mathcal{F}$ is $t$-intersecting if and only if the family $\mathcal{F}^c=\{ [n]\backslash F: F\in \mathcal{F}\}$ is {  $(n-t)$-union}. Since $|\mathcal{F}^c|=|\mathcal{F}|$, the extremal problems for
$t$-intersecting families and $s$-union families can be translated each other.
In 1964, Katona \cite{K64} determined the maximum size of an $s$-union family. 

\begin{theorem}[Katona \cite{K64}] 
\label{thm-K64}
Let $2 \leq s\leq n-2$ and 
$\mathcal{F} \subseteq 2^{[n]}$ be $s$-union. \\
{\rm (i)} If $s=2 d$, then $$|\mathcal{F}| \leq \sum_{0 \leq i \leq d}\binom{n}{i},$$ 
with equality if and only if
$\mathcal{F}=\left\{F\subseteq [n]: |F|\leq d\right\}$. \\ 
{\rm (ii)} If $s=2 d+1$, then 
\[ |\mathcal{F}| \leq \sum_{0 \leq i \leq d}\binom{n}{i}+\binom{n-1}{d}, \] 
with equality if and only if  $\mathcal{F}=\left\{F\subseteq [n]: |F|\leq d\right\} \cup\left\{F \in\binom{[n]}{d+1}: y \in F\right\}$ for some $y \in[n]$.    
\end{theorem}
Katona's bound in (ii) implies Erd\H{o}s--Ko--Rado's bound in (\ref{eq-EKR}). 
The extremal families in Theorem \ref{thm-K64} are known as the Katona families. 
In 2017, Frankl \cite{Fra2017} provided a better bound and determined the sub-optimal $s$-union family by excluding all subfamilies of Katona's families. Recently, the third author and Wu \cite{L24} further determined the third-optimal $s$-union family. 

In this paper, 
we present a sharpening of the {  Katona's theorem} under some additional conditions. For notational convenience, let us first define the extremal families. 
Let $r\geq 1, s\geq 4$ and $n \geq s+2$ be integers. 
Let $D_i=[d]\cup \{d+i\}$ for each $i\in [1,n-d]$. Note that $D_i\cap D_j=[d]$ for $i\neq j$.
For $s=2d$, we define the following families: 
\begin{itemize}
\item 
For $ r \leq d-1$,  
let 
$\mathcal{W}_r(n, 2 d)={[n] \choose \le d-1} \cup\{D_1,\ldots, D_r\} \cup\left\{F \in\binom{[n]}{d}: F \cap D_i \neq \emptyset ~\text{for}~ i\in [r] \right\}$. 

\item 
For $d-1\leq r \leq n-d$, let 
$\mathcal{W}^*(n, 2 d)= {[n] \choose \le d-1} \cup\{ D_1,\ldots ,D_{n-d}\} \cup\left\{F \in\binom{[n]}{d}: F \cap [d] \neq \emptyset\right\}.$

\item 
For $s=6$, let 
$ { \mathcal{W}^{\sharp}(n, 6)}={[n] \choose \le 2}\cup\left\{F\in\binom{[n]}{4}: [2] \subseteq F\right\} \cup\left\{F \in\binom{[n]}{3}: F \cap [2] \neq \emptyset\right\}.$ 
\end{itemize} 
For $s=2d+1$,  we need to introduce the following families first. 
Let $k\geq 3$, $1\leq r\leq k-2$ and $n > 2 k$  be integers. For every $j\in [1,n-k]$, we denote $I_{k,j}:=[2, k]\cup \{k+j\}$. We define  
\begin{equation}
\mathcal{J}_{k,r} :=\{I_{k,1},\ldots, I_{k,r}\}\cup \left\{F \in\binom{[n]}{k}: 1\in F, F \cap I_{k,j}\neq \emptyset, j\in [1,r]\right\} \label{eq-Jkr} 
\end{equation}
and 
\begin{equation}
\mathcal{H}_{k} :=\{I_{k,1},\ldots, I_{k,n-k}\}\cup \left\{F \in\binom{[n]}{k}: 1\in F, F \cap [2,k]\neq \emptyset\right\}. \label{eq-Hk}
\end{equation} 
In addition, let us put
\begin{equation}
    \label{eq-G4}
    \mathcal{G}_{4}:=\left\{F\in\binom{[n]}{4}: 1\notin F, [2,3] \subseteq F\right\}\cup \left\{F \in\binom{[n]}{4}: 1\in F, F\cap[2,3] \neq \emptyset\right\}. 
\end{equation} 
Under the notation in 
(\ref{eq-Jkr}), (\ref{eq-Hk}) and (\ref{eq-G4}), 
we define the following: 
\begin{itemize}
\item 
For $r \leq d-1$, we denote
$\mathcal{W}_r(n, 2 d+1)=\left\{F\subseteq [n]: |F|\leq d\right\}\cup\mathcal{J}_{d+1,r}.
$

\item 
For $d-1\leq r \leq n-d-1$, we denote
$\mathcal{W}^{*}(n, 2 d+1)=\left\{F\subseteq [n]: |F|\leq d\right\}\cup\mathcal{H}_{d+1}.$ 

\item 
For $s=7$, we denote
$ { \mathcal{W}^{\sharp}(n, 7)}=\left\{F\subseteq [n]: |F|\leq 3\right\}\cup\mathcal{G}_4.$
\end{itemize}

We denote $\mathcal{F}_i := \{F\in \mathcal{F}:|F|=i\}$ for each $i\in [n]$.  
The diversity of a family $\mathcal{F}$ is denoted by  
$\gamma (\mathcal{F}) := |\mathcal{F}| - \Delta (\mathcal{F})$, where $\Delta (\mathcal{F})$ 
is the maximum degree of $\mathcal{F}$. In other words, the diversity of $\mathcal{F}$ is defined as the minimum number of $k$-sets in $\mathcal{F}$ such that deleting these $k$-sets results in a star. 
The concept of diversity has proved to be extremely significant in the study of $k$-uniform intersecting families; see, e.g.,  \cite{Kua2018,Huang2019,Fra2020,FK2021,FW2024} 
for recent results.

The second main result of this paper is to characterize the extremal families with the $s$-union property 
under certain conditions. Our result implies some stability results of Katona's theorem, including a result of Frankl \cite{Fra2017} as well as a recent result of the third author and Wu \cite{L24}.

\begin{theorem}\label{main5}
 Let $r\geq 1$ and $4 \leq s \leq n-2$ be integers. Suppose that  $\mathcal{F} \subseteq 2^{[n]}$ is $s$-union. \\ 
 {\rm (i)} If $s=2 d$ for some $d \geq 2$ and $|\mathcal{F} _{d+1}|\geq r$, then for $ r \leq d-1$,
$$
|\mathcal{F}| \leq \sum_{0 \leq i \leq d}\binom{n}{i}-\binom{n-d}{d}+\binom{n-d-r}{d-r}+r, 
$$
where the equality holds if and only if $\mathcal{F} \cong \mathcal{W}_r(n, 2 d)$, or two more possibilities:  when $d=r+1$ and  $\mathcal{F}\cong\mathcal{W}^*(n, 2r+2)$, or when $d=3$ and  $\mathcal{F}\cong { \mathcal{W}^{\sharp}(n, 6)}$; 
For $d\leq r \leq n-d$, 
$$
|\mathcal{F}| \leq \sum_{0 \leq i \leq d}\binom{n}{i}-\binom{n-d}{d}+n-d, 
$$
where the equality  holds if and only if $\mathcal{F}\cong\mathcal{W}^*(n, 2 d)$, or when $d=3$ and $\mathcal{F}\cong { \mathcal{W}^{\sharp}(n, 6)}$. \\ 
{\rm (ii)} If $s=2 d+1$ for some $d \geq 2$ and $\gamma(\mathcal{F} _{d+1})\geq r$, then for $r \leq d-1$, 
$$
|\mathcal{F}| \leq \sum_{0 \leq i \leq d}\binom{n}{i}+\binom{n-1}{d}-\binom{n-d-1}{d}+\binom{n-d-r-1}{d-r}+r, 
$$
where the equality holds if and only if $\mathcal{F} \cong \mathcal{W}_r(n, 2 d+1)$, or two more possibilities: when $d=r+1$ and  $\mathcal{F}\cong\mathcal{W}^*(n, 2r+3)$, or when $d=3$ and  $\mathcal{F}\cong { \mathcal{W}^{\sharp}(n, 7)}$;
For $d\leq r \leq n-d-1$, 
$$
|\mathcal{F}| \leq \sum_{0 \leq i \leq d}\binom{n}{i}+\binom{n-1}{d}-\binom{n-d-1}{d}+n-d-1, 
$$
where the equality  holds if and only if $\mathcal{F}\cong\mathcal{W}^*(n, 2 d+1)$, or when $d=3$ and  $\mathcal{F}\cong { \mathcal{W}^{\sharp}(n, 7)}$.
\end{theorem}

 As an example, we use Theorem \ref{main5} to deduce the stability result of Frankl \cite{Fra2017}. To begin with, let  $\mathcal{K}(n,2d):=\{F\subseteq [n] : |F|\le d\}$ and $\mathcal{K}(n,2d+1):=\big\{F\subseteq [n]: |F|\leq d\big\} \cup 
\big\{F \in\binom{[n]}{d+1}: y \in F\big\}$ (where $y \in[n]$) be 
the extremal families in Theorem \ref{thm-K64}. These families are called the Katona families. 
Setting $r=1$ in Theorem \ref{main5}, we get the following corollary immediately.

\begin{corollary}[Frankl \cite{Fra2017}] 
Suppose that $2\le s\le n-2$ and $\mathcal{F}$ is an $s$-union family. \\ 
{\rm (i)} 
If $\mathcal{F}$ is not a subfamily of $\mathcal{K}(n,2d)$, then  
\[ |\mathcal{F}|\le \sum_{0\le i\le d} {n \choose i} - {n-d-1 \choose d} +1, \]
with equality if and only if 
$\mathcal{F} \cong \mathcal{W}_1(n,2d)$, or when $d=2$ and $\mathcal{F} \cong \mathcal{W}^*(n,4)$.  \\
{\rm (ii)} If $\mathcal{F}$ is not a subfamily of $\mathcal{K}(n,2d+1)$, then  
\[  |\mathcal{F}|\le \sum_{0\le i\le d} {n \choose i} + {n-1 \choose d} - {n-d-2 \choose d} +1,  \]
with equality if and only if $\mathcal{F} \cong \mathcal{W}_1(n,2d+1)$, or when $d=2$ and $\mathcal{F} \cong \mathcal{W^*}(n,5)$. 
\end{corollary}

\subsection{Applications} 

\label{sec-1.3}

As promised, we shall show some quick applications of Theorem \ref{main1} in this section. 
We illustrate that some well-known results 
 due to Frankl, Wang \cite{FW2024} and Kupavskii \cite{K18} could be deduced by applying Theorem \ref{main1}. To begin with, we introduce their results formally. 

\subsubsection{A sharpening of Frankl and Wang}

We begin with the following result on non-empty cross-intersecting families.

\begin{theorem}[Frankl--Wang \cite{FW2024}] \label{F24}
 Let $k \geq 2$, $r\geq 1$ and $n \geq 2 k$ be integers.  Let $\mathcal{F} \subseteq\binom{[n]}{k}$ and $\mathcal{G} \subseteq\binom{[n]}{k}$ be  cross-intersecting families. Suppose that $|\mathcal{G}|\geq |\mathcal{F}| \geq r$.\\
{\rm (i)} If $ r \leq k-1$, then
$$
|\mathcal{F}|+|\mathcal{G}| \leq\binom{n}{k}-\binom{n-k+1}{k}+\binom{n-k-r+1}{k-r}+r.
$$
For $n > 2 k$, the equality holds if and only if $\mathcal{F}=\left\{F_1, \ldots, F_r\right\}$ with $F_i:=[k-1]\cup \{k-1+i\}$ for each $i\in [r]$ and $\mathcal{G}=\left\{G \in\binom{[n]}{k}: G \cap F_i \neq\emptyset~\text{for each}~i\leq r\right\}$, or when $k=r+1$, $\mathcal{F}=\{[r] \cup\{i\}: r+1\le i\le n\}$ and $\mathcal{G}=\left\{G \in\binom{[n]}{r+1}: G \cap[r] \neq \emptyset\right\}$, or when $k=3$, $\mathcal{F}=\mathcal{G}=\left\{G \in\binom{[n]}{3}: 1\in G\right\}$ under isomorphism. \\
{\rm (ii)} If $k\leq r \leq n-k+1$, then
$$
|\mathcal{F}|+|\mathcal{G}| \leq\binom{n}{k}-\binom{n-k+1}{k}+n-k+1.
$$
For $n > 2 k$, the equality holds if and only if $\mathcal{F}=\{[k-1] \cup\{i\}: i \in\{k, \ldots, n\}\}$ and $\mathcal{G}=\left\{G \in\binom{[n]}{k}: G \cap[k-1] \neq \emptyset\right\}$, or when $k=3$, $\mathcal{F}=\mathcal{G}=\left\{G \in\binom{[n]}{3}: 1\in G\right\}$  under isomorphism.
\end{theorem}

Clearly, Theorem \ref{F24} generalizes the Hilton--Milner result (\ref{eq-HM}). We mention that the original proof in \cite{FW2024} did not determine the extremal families of Theorem \ref{F24}.
In this paper, we show that Theorem \ref{F24} is a direct consequence of Theorem \ref{main1}.

\subsubsection{Intersecting family with small diversity} 

We fix the following notation. 
Let $\mathcal{F} \subseteq 2^{[n]}$ be a family. 
For each $i\in [n]$, we denote 
$\mathcal{F}(i)=\left\{F\backslash \{i\}: i \in F \in \mathcal{F}\right\}$ and $\mathcal{F}(\bar{i})=\left\{F \in \mathcal{F}: i \notin F \right\}$. 
Recall that the \textit{diversity} of $\mathcal{F}$ is defined as  
$$\gamma(\mathcal{F}) =\text{min}\{|\mathcal{F}(\bar{i})|: i\in [n]\}.$$
Improving Erd\H{o}s--Ko--Rado's bound in (\ref{eq-EKR}), Hilton and Milner \cite{H67} proved that if $n> 2k$  and $\mathcal{F} \subseteq {[n] \choose k}$ is an intersecting family with 
$\gamma (\mathcal{F})\ge 1$, then $|\mathcal{F}| \le {n-1 \choose k-1} - {n-k -1 \choose k-1} +1$. In 2018, 
Kupavskii \cite{K18} determined the maximum possible size of a $k$-uniform intersecting family with the diversity at least $r$ for every $r\leq n-k$. Let $\mathcal{J}_{k,r}, \mathcal{H}_k$ and $\mathcal{G}_4$ 
be defined in (\ref{eq-Jkr}), (\ref{eq-Hk}) and (\ref{eq-G4}). 

\begin{theorem}[Kupavskii \cite{K18}] 
\label{main4}
Let $k\geq 3, r\geq 1$ and $n > 2 k$ be  integers.  Let $\mathcal{F} \subseteq\binom{[n]}{k}$ be an intersecting family with diversity $\gamma(\mathcal{F}) \geq r$.  \\ 
{\rm (i)} If $ r \leq k-2$, then
$$
|\mathcal{F}| \leq\binom{n-1}{k-1}-\binom{n-k}{k-1}+\binom{n-k-r}{k-r-1}+r,
$$
with equality if and only if $\mathcal{F}\cong\mathcal{J}_{k,r}$, or when $k=r+2$, $\mathcal{F}\cong\mathcal{H}_{r+2}$, or when $k=4$, $\mathcal{F}\cong\mathcal{G}_4$. \\ 
{\rm (ii)}
If $k-1\leq r \leq n-k$, then
$$
|\mathcal{F}| \leq\binom{n-1}{k-1}-\binom{n-k}{k-1}+n-k,
$$
with equality if and only if $\mathcal{F}\cong\mathcal{H}_k$, or when $k=4$, $\mathcal{F}\cong\mathcal{G}_4$.
\end{theorem}

 Recently, Frankl and Wang \cite{FW2024} gave a new proof of Theorem \ref{main4} by using Theorem \ref{F24} together with 
 a classical theorem of Frankl involving the maximum degree (see \cite{F87}).  The proofs in \cite{FW2024, K18} did not characterize the extremal families in Theorem \ref{main4}.
As the second application, we shall give an alternative proof of Theorem \ref{main4} by applying Theorem \ref{main1}.

 \paragraph{Organization.}  
The rest of the paper is organized as follows. 
In Section \ref{sec-3}, we give the proof of Theorem \ref{main1}.
In Section \ref{sec-4}, we illustrate 
the applications of Theorem \ref{main1} and we present short proofs of Theorems \ref{F24} and \ref{main4}.  
In Section \ref{sec-prove-thm1.5}, we provide the proof of Theorem \ref{main5}.

%We conclude this paper with some remarks and an open problem in Section \ref{sec-5}.

\section{Proof of Theorem \ref{main1}} 

\label{sec-3}
To prove Theorem \ref{main1}, we need to  review some fundamental notation and results about the shifting operation.
Let  $\mathcal{F} \subseteq 2^{[n]}$ be a family and $1 \leq i<j \leq n$. The \textit{shifting operator} $s_{i, j}$, discovered by Erdős, Ko and Rado \cite{E61}, is defined as follows:
$$s_{i, j}(\mathcal{F})=\left\{s_{i, j}(F): F \in \mathcal{F}\right\},$$
where
$$
s_{i, j}(F)= \begin{cases}(F \backslash\{j\}) \cup\{i\} & \text { if } j \in F, i \notin F \text { and } (F \backslash\{j\}) \cup\{i\} \notin \mathcal{F}, \\ F & \text { otherwise. }\end{cases}
$$
Obviously, $\left|s_{i, j}(F)\right|=|F|$ and $\left|s_{i, j}(\mathcal{F})\right|=|\mathcal{F}|$.  A frequently used property of $s_{i, j}$ is that it maintains the $t$-intersecting property of a family.

A family $\mathcal{F} \subseteq 2^{[n]}$ is called \textit{shifted} if for all $ F \in \mathcal{F}$,  $i<j$ with $i\notin F$ and $j\in F$, then $(F \backslash\{j\}) \cup\{i\} \in \mathcal{F}$.  
It is well-known that every intersecting family can be transformed to a shifted intersecting family by applying shifting operations repeatedly.  
There are many nice properties on shifted families. For example, 
if $\mathcal{F}$ is a shifted family 
and $\{a_1,\ldots,a_k\}\in \mathcal{F}$ with $a_1<\cdots<a_k$, then for any set  $\{b_1,\ldots,b_k\}$ with $b_1<\cdots<b_k$ and $b_i\leq a_i$ for each $i\in[1,k]$, we have $\{b_1,\ldots,b_k\}\in \mathcal{F}$.

We need the following lemmas.

\begin{lemma}\label{S1}
Let $k\geq 1, t\geq 0$ be integers. Let  $\mathcal{F} \subseteq 2^{[n]}$ be a shifted $(t+1)$-intersecting family. Then $\mathcal{F}(\bar{1})$ is $(t+2)$-intersecting. Moreover,  if $\mathcal{F} \subseteq\binom{[n]}{k+t}$ is a shifted $(t+1)$-intersecting family and $n\geq 2k+t$, then $\mathcal{F}(n)$ is $(t+1)$-intersecting.
\end{lemma}
\begin{proof}
We may assume that $\mathcal{F}(\bar{1})\neq \emptyset$.
Note that $\mathcal{F}(\bar{1})$ is at least $(t+1)$-intersecting. For any $F_1, F_2 \in \mathcal{F}(\bar{1})$, let $j\in F_1 \cap F_2$.
By shiftedness, we have $(F_1\backslash \{j\})\cup\{1\} \in \mathcal{F}$. It follows that $| F_1 \cap F_2|=| \left((F_1\backslash \{j\})\cup\{1\}\right) \cap F_2|+1\geq t+2$. In addition, we may assume that $\mathcal{F}(n)\neq\emptyset$. For any $E_1, E_2 \in \mathcal{F}(n)$, we have $E_1\cup\{n\}, E_2\cup\{n\}\in \mathcal{F}$. Observe that
$
|E_1\cup E_2\cup\{n\}|\leq 2(k+t)-(t+1)\leq n-1.
$
By shiftedness, there exists $x\notin E_1\cup E_2\cup\{n\}$ such that $E_1\cup\{x\} \in \mathcal{F}$. It is immediate that
$
|E_1\cap E_2|=|(E_1\cup\{x\})\cap (E_2\cup\{n\})|\geq t+1,
$
as desired.
\end{proof}

\begin{lemma}\label{S2}
Let $k\ge 1, t\ge 0$, $n\geq 2k+t$ and $1\leq r \leq n-k-t+1$.
Let  $\mathcal{F} \subseteq\binom{[n]}{k+t}$ be shifted and $|\mathcal{F}|\geq r$. 
If $r \leq n-k-t$, then $|\mathcal{F}(\bar{n})|\geq r$.
If $r=n-k-t+1$, then $|\mathcal{F}(\bar{n})|\geq r-1$.
\end{lemma}
\begin{proof}
If $\mathcal{F}(n)=\emptyset$, then $|\mathcal{F}(\bar{n})|=|\mathcal{F}|\geq r$. If $\mathcal{F}(n)\neq\emptyset$, then there exists $F\in \mathcal{F}$ such that $n\in F$. Note that $|F|= k+t$. Let $[n]\backslash F =\{x_1,\ldots, x_{n-k-t}\}$. By shiftedness, $(F\backslash\{n\}) \cup\{x_i\}\in \mathcal{F}(\bar{n})$ for all $i\in[1,n-k-t]$. So $|\mathcal{F}(\bar{n})|\geq n-k-t$, and the result follows.
\end{proof}

\begin{lemma}[See \cite{E61}] \label{S3}
 Let $\mathcal{F} \subseteq 2^{[n]}$ and $\mathcal{G} \subseteq 2^{[n]}$ be  cross-intersecting, and $\mathcal{F}$ be $t$-intersecting.
Then $s_{i, j}(\mathcal{F})$ and $s_{i, j}(\mathcal{G})$ are also cross-intersecting, and $s_{i, j}(\mathcal{F})$ is $t$-intersecting.
\end{lemma}

\begin{lemma}\label{S4}
Let $t\ge 0$, $k\geq 1$ and $n \geq 2 k+t$ be integers.
Let $\mathcal{F} \subseteq\binom{[n]}{k+t}$ and $\mathcal{G} \subseteq\binom{[n]}{k}$ be shifted cross-intersecting families. Then $\mathcal{F}(n)$ and $\mathcal{G}(n)$ are cross-intersecting.
\end{lemma}
\begin{proof}
The case for $\mathcal{F}(n)=\emptyset$ or $\mathcal{G}(n)=\emptyset$ is trivial. If $\mathcal{F}(n)\neq\emptyset$ and $\mathcal{G}(n)\neq\emptyset$, then there are $F_1\in \mathcal{F}(n)$ and $G_1\in \mathcal{G}(n)$ such that $F_1\cup\{n\}\in \mathcal{F}$ and $G_1\cup\{n\}\in \mathcal{G}$.  Observe that
$
|F_1\cup G_1\cup\{n\}|\leq 2k+t-1\leq n-1.
$
By shiftedness, there exists $x\notin F_1\cup G_1\cup\{n\}$ such that $F_1\cup\{x\} \in \mathcal{F}$. Then 
$
|F_1\cap G_1|=|(F_1\cup\{x\})\cap (G_1\cup\{n\})|\geq 1$, as desired.
\end{proof}

\noindent{\bf Proof of Theorem \ref{main1}.}
For fixed $k$ and $t$, we apply induction on  $n\geq 2 k+t$.
First let us consider the base case $n=2 k+t$. If $\mathcal{G}=\emptyset$, then $|\mathcal{F}|+|\mathcal{G}| =|\mathcal{F}|\leq\binom{2k+t}{k}$.
If $\mathcal{G}\neq\emptyset$, then for any $F \in\binom{[2k+t]}{k+t}$, the cross-intersecting property of $\mathcal{F}$ and $\mathcal{G}$ implies that $F \notin \mathcal{F}$ or $[2 k+t] \backslash F \notin \mathcal{G}$. It follows that
$|\mathcal{F}|+|\mathcal{G}| \leq\binom{2k+t}{k},$
as desired in (i) and (ii). 

Next suppose that $n > 2 k+t$ and the result holds for integers less than $n$ and fixed $k$ and $t$. 
By Lemma \ref{S3}, we may assume that $\mathcal{F}$ and $\mathcal{G}$ are shifted.
By Lemma \ref{S2}, we have $|\mathcal{F}(\bar{n})|\geq r$ for $1\leq r \leq n-k-t$ and $|\mathcal{F}(\bar{n})|\geq r-1>k$ for $r= n-k-t+1$.
 Clearly, $\mathcal{F}(\bar{n})$ is $(t+1)$-intersecting. Note that $\mathcal{F}(\bar{n})$ and
$\mathcal{G}(\bar{n})$ are cross-intersecting. The  { inductive hypothesis} can be applied. 

(i) For  $ r \leq k-1$, we have 
\begin{align}\label{f1}
|\mathcal{F}(\bar{n})|+|\mathcal{G}(\bar{n})| \leq\binom{n-1}{k}-\binom{n-k-t}{k}+\binom{n-k-t-r}{k-r}+r. 
\end{align}

(ii) For $k\leq r \leq n-k-t+1$, we get 
\begin{align}\label{f2}
|\mathcal{F}(\bar{n})|+|\mathcal{G}(\bar{n})| \leq\binom{n-1}{k}-\binom{n-k-t}{k}+n-k-t.
\end{align}

We proceed the proof by considering the following two cases.

{\bf Case 1.} Suppose that $|\mathcal{F}(n)|=0$.\vspace{1mm}

Since $|\mathcal{F}(\bar{n})|=|\mathcal{F}|\geq r$, there exist $F_1, \ldots, F_r \in \mathcal{F}(\bar{n})$.
We define the following families:
\begin{align*}
\mathcal{G}_0:=\left\{G \in\binom{[n-1]}{k-1}: G \cap F_1=\emptyset\right\}, 
\end{align*}
and for each $1\leq j\leq r-1$, we define 
\begin{align*}
 \mathcal{G}_j:=\left\{G \in\binom{[n-1]}{k-1}: G \cap F_i \neq \emptyset ~\text{for every $ i\leq j$, and}~ G \cap F_{j+1}=\emptyset\right\}.
\end{align*}
Recall that $F_1,\ldots ,F_r$ are $(k+t)$-sets, we have 
$|\mathcal{G}_{0}|=\binom{n-1-k-t}{k-1}$. For any $1\leq j\leq r-1$, there exist $x_i\in  { F_{i}}, i\in[j]$ such that $x_1,\ldots, x_j\notin F_{j+1}$, {  where $x_1,\ldots, x_j$ do not need to be pairwise different}. Denote $X=\{x_1,\ldots, x_j\}$. Then $1\leq |X|\leq j$ and hence $|\mathcal{G}_j|\geq {n-1-k-t-|X| \choose k-1-|X|}\geq {n-1-k-t-j \choose k-1-j} $.
Observe that $\mathcal{G}_0, \mathcal{G}_1, \ldots, \mathcal{G}_{r-1}$ are pairwise disjoint. In addition, $\mathcal{F}(\bar{n})$ and $\mathcal{G}(n)$ are cross-intersecting.

(i) For $ r \leq k-1$, we have
    \begin{align}
|\mathcal{G}(n)| \leq \binom{n-1}{k-1}-\sum_{j=0}^{r-1}|\mathcal{G}_j| &\leq \binom{n-1}{k-1}-\sum_{j=0}^{r-1}\binom{n-1-k-t-j}{k-1-j} \notag  \\
& \label{eq-explain} =\binom{n-1}{k-1}-\binom{n-k-t}{k-1}+\binom{n-k-t-r}{k-r-1}.
\end{align} 
{We next show that the equality in (\ref{eq-explain}) holds if and only if $\mathcal{F}(\bar{n})=\left\{F_1, \ldots, F_r\right\}$ with $F_i \cap F_j =[k+t-1]$ for every $ i\neq j$,  and $\mathcal{G}(n)=\left\{G \in\binom{[n-1]}{k-1}: G \cap F_i \neq\emptyset ~\text{for each}~ i\leq r\right\}$. Firstly, suppose on the contrary that there exist two distinct sets $F_{i_0},F_{j_0}\in \mathcal{F}(\bar{n})$ such that $|F_{i_0}\cap F_{j_0}| =k+t-s $ for some $s\ge 2$. We can rename the sets of $\mathcal{F}(\bar{n})$ so that we may assume $F_{i_0}=F_1$ and 
$F_{j_0}=F_2$. Then $|\mathcal{G}_1|= {n-1-k-t \choose k-1} - {n-1-k-t-s \choose k-1} 
> {n-1 -k-t-1\choose k-2}$ since $s\ge 2$, which implies that the second inequality in (\ref{eq-explain}) holds strictly, a contradiction. Thus, we get 
$|F_i\cap F_j|= k+t-1$ for every $i\neq j$. 
Next, we need to show that for any $i\neq j$, the intersection $F_i\cap F_j$ are the same $k+t-1$ vertices. Otherwise, without loss of generality, we may assume that 
$F_1=[k+t-1]\cup \{k+t\},F_2=[k+t-1]\cup \{k+t+1\}$ and $F_3=\{x\}\cup [2,k+t]$ for some $x\notin [k+t+1]$. Note that all $(k-1)$-sets that avoid $F_3$ and contain the element $1$ are in $\mathcal{G}_2$. Thus, we have $|\mathcal{G}_2|\ge {n-1 -k-t -1 \choose k-2} > {n-1-k-t -2 \choose k-3}$, so the second inequality in (\ref{eq-explain}) holds strictly, which is a contradiction. }

It follows from (\ref{f1}) and (\ref{eq-explain}) that
\begin{align*}
|\mathcal{F}|+|\mathcal{G}|=|\mathcal{F}(\bar{n})|+|\mathcal{G}(\bar{n})|+|\mathcal{G}(n)|\leq \binom{n}{k}-\binom{n-k-t+1}{k}+\binom{n-k-t-r+1}{k-r}+r.
\end{align*}
The equality holds only if 
the equality in (\ref{eq-explain}) holds. So 
$\mathcal{F}=\left\{F_1,\ldots,F_r \right\}$ with 
$F_i:=[k+t-1]\cup \{k+t-1+i\}$ for each $i\in [r]$  
and $\mathcal{G}=\left\{G \in\binom{[n]}{k}: G \cap F_i \neq\emptyset\text{~for each~} i\in [r]\right\}$.

(ii) For $k\leq r \leq n-k-t+1$, we obtain 
\begin{align}
|\mathcal{G}(n)| \leq \binom{n-1}{k-1}-\sum_{j=0}^{k-1}|\mathcal{G}_j| &\leq \binom{n-1}{k-1}-\sum_{j=0}^{k-1}\binom{n-1-k-t-j}{k-1-j}  \notag \\ 
& =\binom{n-1}{k-1}-\binom{n-k-t}{k-1}. 
\label{eq-Gn}
\end{align}
It follows from (\ref{f2}) and (\ref{eq-Gn}) that
\begin{align*}
|\mathcal{F}|+|\mathcal{G}| =|\mathcal{F}(\bar{n})|+|\mathcal{G}(\bar{n})|+|\mathcal{G}(n)| \leq \binom{n}{k}-\binom{n-k-t+1}{k}+n-k-t,
\end{align*}
so the desired bound in this case holds strictly.

{\bf Case 2.} Suppose that $|\mathcal{F}(n)|\geq 1$.\vspace{1mm}

 By Lemma \ref{S4}, we know that $\mathcal{F}(n)\subseteq\binom{[n-1]}{k+t-1}$ and $\mathcal{G}(n)\subseteq\binom{[n-1]}{k-1}$ are cross-intersecting. By Lemma \ref{S1}, $\mathcal{F}(n)$ is $(t+1)$-intersecting .
Since $n > 2 k+t$ and $k-1\geq 1$, applying Theorem \ref{F16} to $\mathcal{F}(n)$ and $\mathcal{G}(n)$ yields
\begin{align}\label{f3}
|\mathcal{F}(n)|+|\mathcal{G}(n)| \leq\binom{n-1}{k-1}-\binom{n-k-t}{k-1}+1.
\end{align}
{Furthermore, the equality in  (\ref{f3}) holds if and only if $\mathcal{F}(n)=\Big\{[k+t-1]\Big\}$ and $\mathcal{G}(n)=\Big\{G^* \in\binom{[n-1]}{k-1}: G^*\cap $ $ [k+t-1] \neq \emptyset \Big\}$, or $k=3$, there is one more possibility, $\mathcal{F}(n)=\Big\{[t+1] \cup\{i\}: t+2 \le i \le n-1 \Big\}$ and $\mathcal{G}(n)=\left\{G^* \in\binom{[n-1]}{2}: G^* \cap[t+1] \neq \emptyset\right\}$. 
In the former case,  the shiftedness gives 
 $\mathcal{F} =\Big\{[k+t-1]\cup\{i\}:  k+t\le i \le n\Big\}$ and $\mathcal{G} =\left\{G \in\binom{[n]}{k}: G \cap [k+t-1] \neq \emptyset \right\}$. 
In the later case, we have  $\mathcal{F} =\left\{F\in\binom{[n]}{t+3}: [t+1] \subseteq F\right\}$ and $\mathcal{G}=\left\{G \in\binom{[n]}{3}: G \cap[t+1] \neq \emptyset\right\}.$}

(i) For $ r \leq k-1$,  we get from 
(\ref{f1}) and (\ref{f3}) that 
\begin{align*}
|\mathcal{F}|+|\mathcal{G}|=&|\mathcal{F}(\bar{n})|+|\mathcal{G}(\bar{n})|+|\mathcal{F}(n)|+|\mathcal{G}(n)|\\
\leq&\binom{n-1}{k}-\binom{n-k-t}{k}+\binom{n-k-t-r}{k-r}+r+\binom{n-1}{k-1}-\binom{n-k-t}{k-1}+1\\
=&\binom{n}{k}-\binom{n-k-t+1}{k}+\binom{n-k-t-r+1}{k-r}+r+1-\binom{n-k-t-r}{k-r-1}\\
\leq&\binom{n}{k}-\binom{n-k-t+1}{k}+\binom{n-k-t-r+1}{k-r}+r. 
\end{align*}
Moreover, the equality holds if and only if  $1=\binom{n-k-t-r}{k-r-1}$, that is, $k=r+1$, and {  the equality} in (\ref{f3}) holds.  This implies that $\mathcal{F}=\Big\{[t+r] \cup\{i\}: t+r+1\le i\le n \Big\}$ and $\mathcal{G}=\left\{G \in\binom{[n]}{r+1}: G \cap[t+r] \neq \emptyset\right\}$, or one more possibility when $k=3$, $\mathcal{F}=\left\{F\in\binom{[n]}{t+3}: [t+1] \subseteq F\right\}$ and $\mathcal{G}=\left\{G \in\binom{[n]}{3}: G \cap[t+1] \neq \emptyset\right\}$.

(ii) For $k\leq r \leq n-k-t+1$, by (\ref{f2}) and (\ref{f3}), we have 
\begin{align*}
|\mathcal{F}|+|\mathcal{G}|=&|\mathcal{F}(\bar{n})|+|\mathcal{G}(\bar{n})|+|\mathcal{F}(n)|+|\mathcal{G}(n)|\\
\leq&\binom{n-1}{k}-\binom{n-k-t}{k}+n-k-t+\binom{n-1}{k-1}-\binom{n-k-t}{k-1}+1\\
=&\binom{n}{k}-\binom{n-k-t+1}{k}+n-k-t+1.
\end{align*}
The {  above equality } holds only if the equality in (\ref{f3}) holds. Then 
$\mathcal{F}=\{[k+t-1] \cup\{i\}: i \in\{k+t, \ldots, n\}\}$ and $\mathcal{G}=\left\{G \in\binom{[n]}{k}: G \cap[k+t-1] \neq \emptyset\right\}$, or one more possibility when $k=3$, $\mathcal{F}=\left\{F\in\binom{[n]}{t+3}: [t+1] \subseteq F\right\}$ and $\mathcal{G}=\left\{G \in\binom{[n]}{3}: G \cap[t+1] \neq \emptyset\right\}$.

{In the above discussion, we have determined the extremal families 
$\mathcal{F}$ and $\mathcal{G}$ that attain the required upper bound  under the shifting assumption by Lemma \ref{S3}. 
For completeness, we need to characterize the extremal families in the general case. 
Assume that  the extremal families $\mathcal{F}^{\prime}$ and $\mathcal{G}^{\prime}$ satisfy $s_{i, j}\left(\mathcal{F}^{\prime}\right)=\mathcal{F}$ and $s_{i,j}\left(\mathcal{G}^{\prime}\right)=\mathcal{G}$.
Note that $i<j$.
When $\mathcal{F}=\left\{F_1,\ldots,F_r \right\}$ with 
$F_i:=[k+t-1]\cup \{k+t-1+i\}$ for each $i\in [r]$, if $i,j\in [k+t-1]$, then $\mathcal{F}^{\prime}=\mathcal{F}$.
If $i\in [k+t-1]$ and $j\in [k+t, k+t+r-1]$, by symmetry, we assume that $i=1$ and $j=k+t$, then $\mathcal{F}^{\prime}=\{[2,k+t\}\cup\{q\}: q\in\{1\}\cup[k+t+1, k+t+r-1]\}$. Thus $\mathcal{F}^{\prime} \cong\mathcal{F}$. If $i\in [k+t-1]$ and $j\notin [1, k+t+r-1]$, we may assume that $i=1$, then 
$\mathcal{F}^{\prime}=\{[2,k+t-1\}\cup\{q,j\}: q\in[k+t, k+t+r-1]\}$. Thus  $\mathcal{F}^{\prime}\cong\mathcal{F}$. 
If $i,j\in[k+t, k+t+r-1]$, then $\mathcal{F}^{\prime}=\mathcal{F}$.
If $i\in[k+t, k+t+r-1]$ and $j\notin [1, k+t+r-1]$, we may assume that $i=k+t$ and $j=k+t+r$, then $\mathcal{F}^{\prime}=\{[1,k+t-1\}\cup\{q\}: q\in[k+t+1, k+t+r]\}$.  Thus  $\mathcal{F}^{\prime}\cong\mathcal{F}$.
If $i,j\notin [1, k+t+r-1]$, then $\mathcal{F}^{\prime} =\mathcal{F}$. 
Therefore, we conclude that in the case $\mathcal{F}=\left\{F_1,\ldots,F_r \right\}$ with 
$F_i:=[k+t-1]\cup \{k+t-1+i\}$ for each $i\in [r]$, we have $\mathcal{F}^{\prime} \cong\mathcal{F}$. The family $\mathcal{G}^{\prime}$ is just the maximal subfamily of $\binom{[n]}{k}$ that is cross-intersecting with $\mathcal{F}^{\prime}$.
The same is true when $\mathcal{F}$ is the other case.
Let $\mathcal{F}_0$ and $\mathcal{G}_0$ be the original families before applying the shifting operations. 
In other words, $\mathcal{F}$ and $ \mathcal{G}$ are obtained from $\mathcal{F}_0 $ and $ \mathcal{G}_0$ by applying a series of shifting operations.
Thus, we have $\mathcal{F}_0=\mathcal{F}$ and $\mathcal{G}_0=\mathcal{G}$ under isomorphism. } 
$\hfill \square$\vspace{4mm}

\section{Short proofs of Theorems \ref{F24} and \ref{main4}} 

\label{sec-4}

We begin with the following notation.
For $n \geq k+\ell$ and a family $\mathcal{F} \subseteq\binom{[n]}{k}$, we define
$$
\mathcal{D}_{\ell}(\mathcal{F}):=\left\{D \in\binom{[n]}{\ell}: \exists F \in \mathcal{F}, D \cap F=\emptyset\right\}.
$$
Then  $\mathcal{F}$ and $\mathcal{G} \subseteq\binom{[n]}{\ell}$  are cross-intersecting if and only if $\mathcal{F} \cap \mathcal{D}_k(\mathcal{G})=\emptyset$ or equivalently $\mathcal{G} \cap \mathcal{D}_{\ell}(\mathcal{F})=\emptyset$. Moreover, for given $\mathcal{F}$, $\mathcal{G}=\binom{[n]}{\ell} \backslash \mathcal{D}_{\ell}(\mathcal{F})$ is the largest family that is cross-intersecting with $\mathcal{F}$, and for given $\mathcal{G}$, $\mathcal{F}=\binom{[n]}{k} \backslash \mathcal{D}_{k}(\mathcal{G})$ is the largest family that is cross-intersecting with $\mathcal{G}$.

\begin{lemma}[See \cite{F86, M85}] \label{FM}
 Suppose that $n> k+\ell, \mathcal{F} \subseteq\binom{[n]}{k}, |\mathcal{F}|=\binom{n-r}{k-r}$ for some  $ r \leq k$. Then
$
\left|\mathcal{D}_{\ell}(\mathcal{F})\right| \geq\binom{ n-r}{\ell}$, 
with strict inequality unless $ \mathcal{F}=\left\{F \in\binom{[n]}{k}: R \subseteq F\right\}$ for some $R \in \binom{[n]}{r}$.
\end{lemma}

Next we define the lexicographic order on the $k$-element subsets of $[n]$. We say that $F$ is smaller than $G$ in the lexicographic order if $\min \{x: x\in F \backslash G\} 
<\min \{y: y\in  G \backslash F\}$ holds. For $0\leq m\leq \binom{n}{k}$, let $\mathcal{L}(n, k, m)$ be the family of the first $m$ $k$-sets in the lexicographic order. 
The following lemma \cite{H76} will be used in our proof;  
see \cite[p. 266]{FK2017} for a detailed proof.

\begin{lemma}[See \cite{H76}] \label{Ln}
Let $k, \ell, n$ be positive integers with $n> k+\ell$. If $\mathcal{F} \subseteq\binom{[n]}{k}$ and $\mathcal{G} \subseteq\binom{[n]}{\ell}$ are cross-intersecting, then $\mathcal{L}(n, k,|\mathcal{F}|)$ and $\mathcal{L}(n, \ell,|\mathcal{G}|)$ are cross-intersecting.
\end{lemma}

\medskip 
Now we are in a position to prove Theorem \ref{F24}.\vspace{3mm}

\noindent{\bf Proof of Theorem \ref{F24}.}
If $n = 2 k$, then the upper bounds in both  (i) and (ii) are $\binom{2k}{k}$. This bound holds trivially. 
Now we assume that $n > 2 k$.
We denote $\mathcal{F}':=\mathcal{L}(n, k,|\mathcal{F}|)$ and $\mathcal{G}':=\mathcal{L}(n, k,|\mathcal{G}|)$. 
By Lemma \ref{Ln}, we know that 
$\mathcal{F}'$ and $\mathcal{G}'$ are cross-intersecting. Since $|\mathcal{G}'| \geq|\mathcal{F}'| \geq r$, we 
have $\mathcal{F}' \subseteq \mathcal{G}'$, which implies that $\mathcal{F}'$ is intersecting. Next, we prove part (i) only, since part (ii) can be proved similarly. 
In the case $r\le k-1$, 
setting $t=0$ in Theorem \ref{main1} yields 
\begin{equation}
    \label{eq-FW2024} 
    |\mathcal{F}| + |\mathcal{G}| = 
    |\mathcal{F}'| + |\mathcal{G}'| \le {n \choose k} - {n-k+1 \choose k} + {n-k-r +1 \choose k-r} +r.  
\end{equation}
We next characterize the extremal families $\mathcal{F}$ and $\mathcal{G}$ attaining the upper bound.  
By Theorem \ref{main1}, 
the equality in (\ref{eq-FW2024}) holds only if $\mathcal{F}'=\left\{F_1',\ldots,F_r' \right\}$ with 
$F_i':=[k-1]\cup \{k-1+i\}$ for each $i\in [r]$  
and $\mathcal{G}'=\left\{G \in\binom{[n]}{k}: G \cap F_i' \neq\emptyset\text{~for each~} i\in [r]\right\}$, or when $k=r+1$, $\mathcal{F}'=\Big\{[r] \cup\{i\}: i \in\{r+1, \ldots, n\} \Big\}$ and $\mathcal{G}'=\left\{G \in\binom{[n]}{r+1}: G \cap[r] \neq \emptyset\right\}$, or when $k=3$, $\mathcal{F}'=\left\{F\in\binom{[n]}{3}: 1 \in F\right\}$ and $\mathcal{G}'=\left\{G \in\binom{[n]}{3}: 1\in G \right\}$. We need to return to the structure of $\mathcal{F}$ and $\mathcal{G}$. 
In the first extremal case, 
we have $|\mathcal{F}|=|\mathcal{F}'|=r$. 
Denote $\mathcal{F}=\left\{F_1,\ldots,  F_r\right\}$. Similar to the proof of the equality in (\ref{eq-explain}), we can show that $F_i:=[k-1]\cup \{k-1+i\}$ for each $i\in [r]$ and $\mathcal{G}=\left\{G \in\binom{[n]}{k}: G \cap F_i \neq\emptyset ~\text{for each}~ i\leq r\right\}$ under isomorphism. 
In the second extremal case, i.e., $k=r+1$, we have $|\mathcal{F}|=|\mathcal{F}'|=n-r = 
{n-r \choose k-r}$ and $|\mathcal{G}|=|\mathcal{G}'|={n \choose k} - {n-r \choose k}$. 
For a fixed family $\mathcal{F}\subseteq\binom{[n]}{k}$, 
the maximality of $|\mathcal{F}| + |\mathcal{G}|$ yields $\mathcal{G} = \binom{[n]}{k} \setminus \mathcal{D}_k(\mathcal{F})$. So we get  
$|\mathcal{D}_k(\mathcal{F})|={n-r \choose k}$, which implies that the equality case of  Lemma \ref{FM} occurs. 
Thus, it follows from Lemma \ref{FM} that  $\mathcal{F} = \{[r] \cup\{i\}: r+1 \le i \le n\}$ and  $\mathcal{G}= \left\{G \in\binom{[n]}{r+1}: G \cap[r] \neq \emptyset\right\}$ under isomorphism.
In the third extremal case, we have 
$|\mathcal{F}|=|\mathcal{G}|={n-1 \choose 2}$. Similarly, we get  $| \mathcal{D}_{3}(\mathcal{F})|=\binom{n-1}{3}$. By Lemma \ref{FM}, we obtain $\mathcal{F}= \left\{F\in\binom{[n]}{3}: 1\in F\right\}$ and 
$\mathcal{G}= \left\{G\in\binom{[n]}{3}: 1\in G\right\}$ under isomorphism.  
$\hfill \square$\vspace{3mm}

Excluding the extremal families in Theorem \ref{F24}, we obtain the following result. 

\begin{lemma}\label{main3}
Let $k\geq 2, r \geq 1$ and $n \geq 2 k+1$ be integers.   Let $\mathcal{F} \subseteq\binom{[n]}{k}$ and $\mathcal{G} \subseteq\binom{[n]}{k}$ be cross-intersecting families.
Suppose that $|\mathcal{F}| \geq r$, $|\mathcal{G}| \geq r$ and $|\mathcal{F} \cap \mathcal{G}|\leq r-1$. 
\begin{itemize}
\item[{\rm (i)}] If $r \leq k-1$, then
$
|\mathcal{F}|+|\mathcal{G}| \leq\binom{n}{k}-\binom{n-k+1}{k}+\binom{n-k-r+1}{k-r}+r-1.
$ 

\item[{\rm (ii)}]
If $k\leq r \leq n-k+1$, then
$
|\mathcal{F}|+|\mathcal{G}| \leq\binom{n}{k}-\binom{n-k+1}{k}+n-k.
$
\end{itemize}
\end{lemma}

 For a family $\mathcal{F} \subseteq 2^{[n]}$ and $1\leq i\neq j\leq n$, we denote
\begin{align*}
&\mathcal{F}(i,j)=\left\{F\backslash \{i,j\}: i, j\in F \in \mathcal{F}\right\},\\
&\mathcal{F}(i,\bar{j})=\mathcal{F}(\bar{j},i)=\left\{F\backslash \{i\}: i \in F, j\notin F, F \in \mathcal{F}\right\},\\
&\mathcal{F}(\bar{i},\bar{j})=\left\{F: i, j\notin F \in \mathcal{F}\right\}.
\end{align*}
Next we prove Theorem \ref{main4} by employing Theorem \ref{main1} and Lemma \ref{main3}.\vspace{3mm}

\noindent{\bf Proof of Theorem \ref{main4}.}
It is well-known that 
applying the shifting operations preserves the intersecting property. 
So we consider the following two cases: 
(i) After applying all possible shifting operations, we arrive at a shifted intersecting family $\mathcal{F}^*$ such that $\gamma (\mathcal{F}^*)\ge r$; (ii) After applying some shifting operations, we get a family $\mathcal{G}$, while at some step $s_{i,j}$, the resulting family $s_{i,j}(\mathcal{G})$ violates the diversity condition, i.e., $\gamma (s_{i,j}(\mathcal{G}))\le r-1$.  

{\bf Case 1.} Suppose that  the shifting operations end up with a shifted intersecting family $\mathcal{F}^*$ satisfying $|\mathcal{F}^*|=|\mathcal{F}|$ and $\gamma(\mathcal{F}^*) \geq r$. 
Then we have $|\mathcal{F}^*(\bar{1})|\geq r\geq 1$. Since $\mathcal{F}^*$ is intersecting, Lemma \ref{S1} implies that $\mathcal{F}^*(\bar{1})$ is $2$-intersecting. Note that $\mathcal{F}^*(\bar{1}) \subseteq\binom{[n]\backslash \{1\}}{k}$ and $\mathcal{F}^*(1) \subseteq\binom{[n]\backslash \{1\}}{k-1}$ are  cross-intersecting. Thus, Theorem \ref{main1} could be applied. 

(i) For $r\le k-2$, the case $t=1$ in Theorem \ref{main1} yields 
\[  |\mathcal{F}^*|=|\mathcal{F}^*(\bar{1})|+|\mathcal{F}^*(1)| \le 
{n-1 \choose k-1} - {n-k \choose k-1} + 
{n-k-r \choose k-1-r} + r,
\]
where the equality holds if and only if $\mathcal{F}^*=\mathcal{J}_{k,r}$, or when $k=r+2$ and $\mathcal{F}^*=\mathcal{H}_{r+2}$, or when $k=4$ and $\mathcal{F}^*=\mathcal{G}_4$. Hence, we can verify that $\mathcal{F} $ is isomorphic to 
$\mathcal{J}_{k,r}$ or $\mathcal{H}_{r+2}$ or 
$\mathcal{G}_4$. 

(ii) For $k-1\le r \le n-k$, by Theorem \ref{main1} again, we get 
\[  |\mathcal{F}^*|=|\mathcal{F}^*(\bar{1})|+|\mathcal{F}^*(1)| \le 
{n-1 \choose k-1} - {n-k \choose k-1} + n-k,  \]
with equality if and only if $\mathcal{F}^*=\mathcal{H}_k$, or when $k=4$ and $\mathcal{F}^*=\mathcal{G}_4$. Similarly, we can check that $\mathcal{F}$ is isomorphic to $\mathcal{H}_k$ or $\mathcal{G}_4$. 

{\bf Case 2.} 
 Suppose that there is a family $\mathcal{G} \subseteq\binom{[n]}{k}$ obtained from $\mathcal{F}$ by repeatedly shifting operations such that $\gamma(\mathcal{G}) \geq r$, $|\mathcal{G}|=|\mathcal{F}|$, but after some shifting $s_{i,j}$, we have $\gamma(s_{i, j}(\mathcal{G})) \leq r-1$.\vspace{1mm}

{\bf Subcase 2.1.} Suppose that $\mathcal{G}(\bar{i},\bar{j})= \emptyset$.\vspace{1mm}

First note that $\gamma(\mathcal{G}) \geq r$ implies that $|\mathcal{G}(i,\bar{j})|=|\mathcal{G}(\bar{j})|\geq r$ and $|\mathcal{G}(\bar{i},j)|=|\mathcal{G}(\bar{i})|\geq r$.  
We see that $\mathcal{G}(i,\bar{j})$ and $\mathcal{G}(\bar{i},j)$ are cross-intersecting. From $\gamma(s_{i, j}(\mathcal{G})) \leq r-1$, we get  $|\mathcal{G}(\bar{i},j)\cap \mathcal{G}(i,\bar{j})|\leq r-1$. 
In this case,  Lemma \ref{main3} is applicable. 

(i) For $ r \leq k-2$, 
by Lemma \ref{main3}, we have
$$
|\mathcal{G}(i,\bar{j})|+|\mathcal{G}(\bar{i}, j)| \leq\binom{n-2}{k-1}-\binom{n-k}{k-1}+\binom{n-k-r}{k-r-1}+r-1. 
$$

(ii) For $k-1\leq r \leq n-k$, we get 
$$
|\mathcal{G}(i,\bar{j})|+|\mathcal{G}(\bar{i}, j)|  \leq\binom{n-2}{k-1}-\binom{n-k}{k-1}+n-k-1. 
$$
Combining the inequality $|\mathcal{G}(i,j)|\leq \binom{n-2}{k-2}$ and the fact $|\mathcal{F}|=|\mathcal{G}(i,j)|+|\mathcal{G}(i,\bar{j})|+|\mathcal{G}(\bar{i}, j)|$, we get that in both two cases, $|\mathcal{F}|$ is {\it strictly} smaller than the required upper bound.

{\bf Subcase 2.2.} Suppose that $\mathcal{G}(\bar{i},\bar{j})\neq \emptyset$.\vspace{1mm}

 We denote $|\mathcal{G}(\bar{i},\bar{j})|:=s\ge 1$. By $\gamma(s_{i, j}(\mathcal{G})) \leq r-1$, we have $ s \leq r-1$ and $|\mathcal{G}(\bar{i},j)\cap \mathcal{G}(i,\bar{j})|\leq r-1-s$. In addition, $\gamma(\mathcal{G}) \geq r$ implies that $|\mathcal{G}(i,\bar{j})|\geq r-s$ and $|\mathcal{G}(\bar{i},j)|\geq r-s$.

(i) For $r \leq k-2$, we have $1\leq r-s < k-2$. Recall that $\mathcal{G}(i,\bar{j})$ and $\mathcal{G}(\bar{i},j)$ are 
cross-intersecting. It follows from Lemma \ref{main3} that
\begin{align}\label{a1}
|\mathcal{G}(i,\bar{j})|+|\mathcal{G}(\bar{i}, j)| \leq\binom{n-2}{k-1}-\binom{n-k}{k-1}+\binom{n-k-r+s}{k-1-r+s}+r-s-1.
\end{align}
Note that $\mathcal{G}(i,j)$ and $\mathcal{G}(\bar{i},\bar{j})$ are cross-intersecting and $|\mathcal{G}(\bar{i},\bar{j})|=s\leq r-1<k-2$. Therefore, similar to (\ref{eq-explain}), we have
\begin{align}
|\mathcal{G}(i,j)| &\leq \binom{n-2}{k-2}-\sum_{j=0}^{s-1}\binom{n-2-k-j}{k-2-j} \notag \\
&=\binom{n-2}{k-2}-\binom{n-k-1}{k-2}+\binom{n-k-s-1}{k-s-2}.\label{a2}
\end{align}
Consequently, we obtain 
\begin{align*}
|\mathcal{F}|&=|\mathcal{G}(i,j)|+|\mathcal{G}(i,\bar{j})|+|\mathcal{G}(\bar{i}, j)|+|\mathcal{G}(\bar{i},\bar{j})|\\
&\leq \binom{n-1}{k-1}-\binom{n-k}{k-1}-\binom{n-k-1}{k-2}+\binom{n-k-s-1}{k-s-2}+\binom{n-k-r+s}{k-1-r+s}+r-1.
\end{align*}
Let $f(s):=\binom{n-k-s-1}{k-s-2}+\binom{n-k-r+s}{k-1-r+s}$, where $1\leq s \leq r-1$. It is easy to see that $\text{max}_{1\leq s \leq r-1} f(s)=\text{max}\{f(1), f(r-1)\}=f(r-1)=\binom{n-k-1}{k-2}+\binom{n-k-r}{k-r-1}$. So
$$
|\mathcal{F}| <\binom{n-1}{k-1}-\binom{n-k}{k-1}+\binom{n-k-r}{k-r-1}+r.
$$
So the desired upper bound in (i) holds strictly. 

(ii) For $k-1\leq r \leq n-k$,  
the arguments in this case are slightly complicated.

If $1\leq r-s \leq k-2$, equivalently, $r-k+2\leq s \leq r-1$, then (\ref{a1}) holds. 

If $k-1\leq r-s \leq r-1<n-k$, that is, $1\leq s \leq r-k+1$, then Lemma \ref{main3} gives 
\begin{equation}
    \label{eq-para-3-2}
    |\mathcal{G}(i,\bar{j})|+|\mathcal{G}(\bar{i}, j)| \leq\binom{n-2}{k-1}-\binom{n-k}{k-1}+n-k-1. 
\end{equation}

If $1\leq s\leq k-2$, then (\ref{a2}) holds. 

If $k-1\leq s\leq r-1$, then we similarly have 
\begin{equation}
    \label{eq-para-3-3}
    |\mathcal{G}(i,j)|\leq \binom{n-2}{k-2}-\sum_{j=0}^{k-2}\binom{n-2-k-j}{k-2-j}=\binom{n-2}{k-2}-\binom{n-k-1}{k-2}. 
\end{equation}

{ {\bf Case 2.2.1.}} Let us first consider the case $k-1\leq r-k+2$. We have  three subcases: \\ 
For $r-k+2\leq s \leq r-1$, we obtain from (\ref{a1}) and (\ref{eq-para-3-3}) that 
\begin{align} 
|\mathcal{F}|&\le  \binom{n-2}{k-1}-\binom{n-k}{k-1}+\binom{n-k-r+s}{k-1-r+s}+r-s-1 + 
\binom{n-2}{k-2}-\binom{n-k-1}{k-2}+s \notag \\
&= \binom{n-1}{k-1}-\binom{n-k}{k-1}-\binom{n-k-1}{k-2}+\binom{n-k-r+s}{k-1-r+s}+r-1 \notag \\
&\leq \binom{n-1}{k-1}-\binom{n-k}{k-1}+r-1 \notag \\ 
&<\binom{n-1}{k-1}-\binom{n-k}{k-1}+n-k. \label{a3}
\end{align}
For $k-2\leq s \leq r-k+1$, we obtain from 
(\ref{a2}) and (\ref{eq-para-3-3}) that 
$|\mathcal{G}(i,j)|\le {n-2 \choose k-2} - {n-k-1 \choose k-2} +1$. Combining with 
(\ref{eq-para-3-2}), we have 
\begin{align*}
|\mathcal{F}|&\leq\binom{n-2}{k-2}-\binom{n-k-1}{k-2}+1+\binom{n-2}{k-1}-\binom{n-k}{k-1}+n-k-1+s\\
&= \binom{n-1}{k-1}-\binom{n-k}{k-1}-\binom{n-k-1}{k-2}+n-k+s \\ 
& <\binom{n-1}{k-1}-\binom{n-k}{k-1}+n-k.
\end{align*}
For $1\leq s \leq k-3$, it follows from 
(\ref{a2}) and (\ref{eq-para-3-2}) that 
\begin{align}
|\mathcal{F}|&\leq \binom{n-2}{k-2}-\binom{n-k-1}{k-2}+\binom{n-k-s-1}{k-s-2}+\binom{n-2}{k-1}-\binom{n-k}{k-1}+n-k-1+s \notag \\
&= \binom{n-1}{k-1}-\binom{n-k}{k-1}-\binom{n-k-1}{k-2}+\binom{n-k-s-1}{k-s-2}+n-k-1+s \notag \\
&\leq \binom{n-1}{k-1}-\binom{n-k}{k-1}-\binom{n-k-1}{k-2}+\binom{n-k-2}{k-3}+n-k \notag \\
&<\binom{n-1}{k-1}-\binom{n-k}{k-1}+n-k. 
\label{a4} 
\end{align}

{ {\bf Case 2.2.2.}} Next we consider the case $k-1>r-k+2$. 
We also have three subcases: \\ 
For $k-1\leq s \leq r-1$, we can see that (\ref{a3}) holds similarly. \\ 
For $r-k+2\leq s \leq k-2$, we have
\begin{align*}
|\mathcal{F}|\leq&\binom{n-2}{k-2}-\binom{n-k-1}{k-2}+\binom{n-k-s-1}{k-s-2}+\binom{n-2}{k-1}-\binom{n-k}{k-1}\\
&+\binom{n-k-r+s}{k-1-r+s}+r-s-1+s\\
=& \binom{n-1}{k-1}-\binom{n-k}{k-1}-\binom{n-k-1}{k-2}+\binom{n-k-s-1}{k-s-2}+\binom{n-k-r+s}{k-1-r+s}+r-1,
\end{align*}
Recall that $f(s)=\binom{n-k-s-1}{k-s-2}+\binom{n-k-r+s}{k-1-r+s}$, where $r-k+2\leq s \leq k-2$. Since $k-1>r-k+2$, we get $r<2k-3$ and then $r-k+2\leq\frac{r}{2}\leq k-2$. Thus, we infer that $\text{max}_{r-k+2\leq s \leq k-2} f(s)=\text{max}\{f(r-k+2), f(k-2)\}=f(k-2)=\binom{n-r-2}{2k-r-3}+1$. As a consequence, we obtain 
\begin{align*}
|\mathcal{F}|\leq&\binom{n-1}{k-1}-\binom{n-k}{k-1}-\binom{n-k-1}{k-2}+\binom{n-r-2}{2k-r-3}+r\\
\leq&\binom{n-1}{k-1}-\binom{n-k}{k-1}-\binom{n-k-1}{k-2}+\binom{n-k-1}{k-2}+k-1\\
<&\binom{n-1}{k-1}-\binom{n-k}{k-1}+n-k.
\end{align*}
For $1\leq s \leq r-k+1$, it follows that (\ref{a4}) holds.

We conclude that in Case 2, $|\mathcal{F}|$ is strictly smaller than the required upper bound.  $\hfill \square$\vspace{3mm}

\section{Proof of Theorem \ref{main5}}

\label{sec-prove-thm1.5}

Recall that $\mathcal{F}_i := \{F\in \mathcal{F}:|F|=i\}$. 
In order to prove Theorem \ref{main5}, we need the following lemma, which was proved by Katona 
\cite{K64}; see \cite{Fra2017} for a self-contained and streamlined proof.

\begin{lemma}[Katona's inequality \cite{K64,Fra2017}] 
\label{51}
 Let $2 \leq s \leq n-2$ be an integer. Suppose that  $\mathcal{F} \subseteq 2^{[n]}$ is an $s$-union family. Then for all $1\leq i\leq \frac{s}{2}$, 
$$
|\mathcal{F}_i|+|\mathcal{F}_{s+1-i}| \leq \binom{n}{i},
$$
where the equality holds if and only if $\mathcal{F}_i=\binom{[n]}{i}$ and 
$\mathcal{F}_{s+1-i}=\emptyset$.
\end{lemma}

\noindent{\bf Proof of Theorem \ref{main5}.}
(i) We consider the case $s=2d$. Note that
$$
|\mathcal{F}| =|\mathcal{F}_0|+\sum_{1 \leq i \leq d} \Big(|\mathcal{F}_i|+|\mathcal{F}_{2d+1-i}|\Big).
$$
Since $\mathcal{F}_{d+1}\neq \emptyset$ and $\mathcal{F}$ is $2d$-union, 
we know that $\mathcal{F}_{d+1}$ is $2$-intersecting. Observe that $\mathcal{F}_{d}$ and $\mathcal{F}_{d+1}$ are cross-intersecting. 
So we can apply the case $t=1$ in Theorem \ref{main1}.

For $r\le d-1$, we get 
\[  |\mathcal{F}_d| + |\mathcal{F}_{d+1}| \le 
{n \choose d} - {n-d \choose d} + {n-d-r \choose d-r} +r. \]

For $d\le r \le n-d$, we have 
\[   |\mathcal{F}_d| + |\mathcal{F}_{d+1}| \le 
{n \choose d} - {n-d \choose d} + n-d. \]
Moreover, we know from 
Lemma \ref{51} that $|\mathcal{F}_i| + |\mathcal{F}_{2d+1-i}| \le {n \choose i}$ 
for every $ i\in [1, d-1]$. 
So we immediately obtain the desired upper bound. 

(ii) For the case $s=2d+1$, we can see that  
$$
|\mathcal{F}| =|\mathcal{F}_0|+\sum_{1 \leq i \leq d}\Big(|\mathcal{F}_i|+|\mathcal{F}_{2d+2-i}|\Big)+|\mathcal{F}_{d+1}|.
$$
Since $\mathcal{F}$ is $(2d+1)$-union,  $\mathcal{F}_{d+1}$ is intersecting. 
By Theorem \ref{main4}, we get that for $r\le d-1$, 
\[ |\mathcal{F}_{d+1}| \le {n-1 \choose d}  
 -{n-d-1 \choose d} + {n-d-r-1 \choose d-r} +r. \]
 For $d\le r\le n-d-1$, we have 
\[  |\mathcal{F}_{d+1}| \le {n-1 \choose d} - {n-d-1 \choose d} + n-d-1. \]
Thus, the desired upper bound follows from Lemma \ref{51}.
$\hfill \square$\vspace{3mm}

\noindent 
{\bf Remark.} 
Note that the bound in Lemma \ref{51} can be improved whenever $|\mathcal{F}_{s+1-i}|\ge 1$. 
Indeed, if $\mathcal{F}_{s+1-i}\neq\emptyset$,  then $\mathcal{F}_{s+1-i}$ is $(s+2-2i)$-intersecting, and we can apply Theorem \ref{main1} to $\mathcal{F}_i$ and $\mathcal{F}_{s+1-i}$. Thus, the bounds in Theorem \ref{main5} can be improved under certain conditions.

\iffalse 
\section*{Declaration of competing interest}
We declare that we have no conflict of interest to this work.

\section*{Data availability}
No data was used for the research described in the article.
\fi 

\section*{Acknowledgement} 
The authors would like to express their sincere thanks to the referee for the valuable suggestions which improved the presentation of the manuscript. Yongtao Li would also like to thank Biao Wu for the inspiring discussions. 
Lihua Feng was supported by 
the NSFC (Nos. 12271527 and 12471022). 
Yongtao Li was supported by the Postdoctoral Fellowship Program of CPSF (No. GZC20233196). 
This work was also partially supported by 
the NSF of Qinghai Province (No. 2025-ZJ-902T).  
This paper is equally contributed.


\begin{thebibliography}{99}
\bibitem{AS2016}
N. Alon, J.H. Spencer, 
The Probabilistic Method, Fourth edition, John Wiley \& Sons, Inc., Hoboken, NJ, 2016.

\bibitem{BF2022}
P. Borg, C. Feghali, 
The maximum sum of sizes of cross-intersecting families of subsets of a set, 
Discrete Math. 345 (2022), No. 112981.

\bibitem{CLLW2022}
M. Cao, M. Lu, B. Lv, K. Wang, 
Nearly extremal non-trivial cross $t$-intersecting families and $r$-wise $t$-intersecting families, 
European J. Combin. 120 (2024), No. 103958. 

\bibitem{Day1974b}
D.E. Daykin, Erd\H{o}s--Ko--Rado from Kruskal--Katona, 
J. Combin. Theory, Ser. A 17 (1974) 254--255.

\bibitem{E61}P. Erd\H{o}s, C. Ko, R. Rado, 
Intersection theorems for systems of finite sets, 
Q. J. Math. Oxford 2 (1961) 313--320.


%\bibitem{F19}P. Frankl, A simple proof of the Hilton--Milner theorem,  Mosc. J. Comb. Number Theory 8 (2019) 97--101. 

%\bibitem{FF1986} 
%P. Frankl, Z. F\"{u}redi, Non-trivial intersecting families, J. Combin. Theory, Ser. A 41 (1986) 150--153. 


\bibitem{F87}
P. Frankl, 
Erd\H{o}s-Ko-Rado theorem with conditions on the maximal degree, 
J. Combin. Theory, Ser. A 46 (1987) 252--263.

\bibitem{F16}
P. Frankl, 
New inequalities for cross-intersecting families, 
Mosc. J. Comb. Number Theory 6 (2016) 27--32.


\bibitem{Fra2017} 
P. Frankl, 
A stability result for the Katona theorem, 
J. Combin. Theory, Ser. B 122 (2017) 869--876.

\bibitem{Fra2020}
P. Frankl, 
Maximum degree and diversity in intersecting hypergraphs, 
J. Combin. Theory, Ser. B 144 (2020) 81–94

 \bibitem{Fra2024}
 P. Frankl, 
 On the maximum of the sum of the sizes of non-trivial cross-intersecting families, 
 Combinatorica 44  (2024) 15--35.  
 


\bibitem{FF2012} 
P. Frankl, Z. F\"{u}redi, 
A new short proof of the EKR theorem, 
J. Combin. Theory, Ser. A 119 (2012) 1388--1390. 



%\bibitem{FT16}P. Frankl, N. Tokushige, Invitation to intersection problems for finite sets, J. Combin. Theory, Ser. A 144 (2016) 157--211. 

\bibitem{FK2017}
P. Frankl, A. Kupavskii, 
A size-sensitive inequality for cross-intersecting families, 
European J. Combin. 62 (2017) 263--271.

\bibitem{FK2017-CPC}
P. Frankl, A. Kupavskii, 
Uniform $s$-cross-intersecting families, 
Combin. Probab. Comput.  26 (2017) 517--524. 

\bibitem{FK2021}
P. Frankl, A. Kupavskii, Diversity, 
J. Combin. Theory, Ser. A 182 (2021), No. 105468

 \bibitem{FT92}
P. Frankl, N. Tokushige, 
Some best possible inequalities concerning cross-intersecting families, 
J. Combin. Theory, Ser. A 61 (1992) 87--97.

\bibitem{FT1998}
P. Frankl, N. Tokushige, 
Some inequalities concerning cross-intersecting families, Combin. Probab. Comput. 7 (1998) 247--260. 

\bibitem{FW2024-cover3}
P. Frankl, J. Wang, 
Intersecting families with covering number three, 
J. Combin. Theory, Ser. B 171 (2025) 96--139. 


\bibitem{FW2024}
P. Frankl, J. Wang, 
Improved bounds on the maximum diversity of intersecting families, 
European J. Combin. 118 (2024), No. 103885.



\bibitem{F86}Z. F{\"u}redi, J.R. Griggs, Families of finite sets with minimum shadows, Combinatorica 6 (1986) 355--363.


\bibitem{Fur2006}Z. F\"{u}redi, K.-W. Hwang, P.M. Weichsel, 
A proof and generalizations of the Erd\H{o}s–Ko–Rado theorem 
using the method of linearly independent polynomials, 
in: Topics in Discrete Mathematics, 
in: Algorithms Combin., vol. 26, Springer, Berlin,
2006, pp. 215--224.  

%\bibitem{H17}J. Han, Y. Kohayakawa, The maximum size of a non-trivial intersecting uniform family that is not a subfamily of the Hilton-Milner family, Proc. Amer. Math. Soc. 145 (2017) 73--87.


\bibitem{H76}A.J.W. Hilton, The Erd\H{o}s-Ko-Rado theorem with valency conditions, in: Unpublished Manuscript, 1976.

\bibitem{H67} A.J.W. Hilton, E.C. Milner, Some intersection theorems for systems of finite sets, 
Q. J. Math. 18 (1967) 369--384.

\bibitem{Huang2019}
H. Huang, 
Two extremal problems on intersecting families, 
European J. Combin. 76 (2019) 1--9. 

%\bibitem{H24} Y. Huang, Y.J. Peng, Stability of intersecting families, European J. Combin. 115 (2024) 103774.

\bibitem{HZ2017}H. Huang, Y. Zhao, 
Degree versions of the Erd\H{o}s--Ko--Rado theorem 
and Erd\H{o}s hypergraph matching conjecture, 
J. Combin. Theory, Ser. A 150 (2017) 233--247. 

\bibitem{HP2025-jcta}
Y. Huang,  Y. Peng, 
Non-empty pairwise cross-intersecting families, 
J. Combin. Theory, Ser. A 211 (2025), No. 105981. 

\bibitem{H18}G. Hurlbert, V. Kamat, 
New injective proofs of the Erd\H{o}s--Ko--Rado and Hilton--Milner
theorems, Discrete Math. 341 (2018) 1749--1754.

\bibitem{K64}G.O.H. Katona, Intersection theorems for systems of finite sets, Acta Math. Hungar. 15 (1964) 329--337.

\bibitem{Kat1972}G.O.H. Katona, A simple proof of the Erd\H{o}s--Chao Ko--Rado theorem, 
J. Combin. Theory, Series B 13 (1972) 183--184.

%\bibitem{KM17}A. Kostochka, D. Mubayi, The structure of large intersecting families, Proc. Amer. Math. Soc. 145 (2017) 2311--2321. 

\bibitem{K18}A. Kupavskii, Structure and properties of large intersecting families, (2018) arXiv:1810.00920.

\bibitem{Kua2018}
A. Kupavskii, 
Diversity of uniform intersecting families, 
European J. Combin. 74 (2018) 39--47. 

\bibitem{KZ2018} 
A. Kupavskii, D. Zakharov, 
Regular bipartite graphs and intersecting families, 
J. Combin. Theory, Ser. A 155 (2018) 180--189.

\bibitem{L24}
Y. Li, B. Wu, Stabilities for non-uniform $t$-intersecting families, 
Electron. J. Combin. 31 (4) (2024), \#P4.3.

\bibitem{Lov1979}L. Lov\'{a}sz, On the Shannon capacity of a graph,  
IEEE Trans. Inform. Theory 25 (1979)  1--7.  

\bibitem{M85}M. M{\"o}rs, A Generalization of a Theorem of Kruskal, Graphs and Combinatorics 1 (1985) 167--183.

\bibitem{SW2021}
A. Scott, E. Wilmer, 
Combinatorics in the exterior algebra and the Bollobás two families theorem, 
J. Lond. Math. Soc.  104 (2021) 1812--1839.

\bibitem{SFQ2022}
C. Shi,  P. Frankl, J. Qian, 
 On non-empty cross-intersecting families, 
 Combinatorica 42 (2022) 1513--1525.


\bibitem{WZ2013}
J. Wang, H. Zhang, 
Nontrivial independent sets of bipartite graphs and cross-intersecting families,
J. Combin. Theory, Ser. A 120 (2013) 129--141.


\bibitem{W23}B. Wu, A refined result on cross-intersecting families, Discrete Appl. Math. 339 (2023) 149--153.


\bibitem{YKXZG2022}
W. Yu, X. Kong, Y. Xi, X. Zhang, G. Ge, 
Bollobás-type theorems for hemi-bundled two families, 
European J. Combin. 100 (2022), No. 103438. 

%\bibitem{W240}
%Y. Wu, Y. Li, L. Feng,  J. Liu, G. Yu, Stabilities of intersecting families revisited, (2024), arXiv:2411.03674. 

\end{thebibliography}
\end{document}